\documentclass[reqno]{amsart}
\usepackage{amssymb}
\usepackage{enumerate}
\usepackage[mathscr]{euscript}
\usepackage[hidelinks]{hyperref}
\usepackage{cleveref}
\usepackage{mathtools}
\usepackage{tikz-cd}

\numberwithin{equation}{section}

\newtheorem{thm}[equation]{Theorem}
\newtheorem*{thm*}{Theorem}
\newtheorem{prop}[equation]{Proposition}
\newtheorem{cor}[equation]{Corollary}

\theoremstyle{definition}
\newtheorem{defn}[equation]{Definition}
\newtheorem{ex}[equation]{Example}
\newtheorem{rec}[equation]{Recollection}
\newtheorem{rem}[equation]{Remark}
\newtheorem{cons}[equation]{Construction}
\newtheorem{summ}[equation]{Summary}
\newtheorem*{Ack}{Acknowledgements}

\newcommand{\nc}{\newcommand}
\nc{\dmo}{\DeclareMathOperator}

\nc{\gS}{\Sigma}

\nc{\bbZ}{\mathbb{Z}}

\nc{\mcO}{\mathcal{O}}

\nc{\rmb}{\mathrm{b}}
\nc{\rmc}{\mathrm{c}}
\nc{\rmD}{\mathrm{D}}
\nc{\rmK}{\mathrm{K}}
\nc{\rmL}{\mathrm{L}}
\nc{\rmR}{\mathrm{R}}
\nc{\rmS}{\mathrm{S}}

\nc{\scA}{\mathscr{A}}
\nc{\scC}{\mathscr{C}}
\nc{\scD}{\mathscr{D}}
\nc{\scK}{\mathscr{K}}
\nc{\scT}{\mathscr{T}}

\dmo{\Char}{char}
\dmo{\Coh}{Coh}
\dmo{\Free}{Free}
\dmo{\Hom}{Hom}
\dmo{\Id}{Id}
\dmo{\im}{Im}
\dmo{\Inj}{Inj}
\dmo{\Ker}{Ker}
\dmo{\Mod}{Mod}
\dmo{\noeth}{noeth}
\dmo{\QCoh}{QCoh}
\dmo{\res}{res}
\dmo{\smod}{mod}
\dmo{\sMod}{\ul{\Mod}}
\dmo{\Spec}{Spec}
\dmo{\Supp}{Supp}
\dmo{\thick}{thick}

\nc{\ol}{\overline}
\nc{\ot}{\otimes}
\nc{\ul}{\underline}
\nc{\wt}{\widetilde}
\nc{\xr}{\xrightarrow}

\nc{\rcolon}{\nobreak \mskip 6muplus1mu\mathpunct {}\nonscript \mkern -\thinmuskip {:}\mskip 2mu\relax}

\begin{document}
\title{The singularity category of a separable extension}
\author{Charalampos Verasdanis}
\address{Institute of Mathematics, Czech Academy of Sciences, \v{Z}itn\'a 25, 115 67 Prague, Czech Republic}
\email{verasdanis@math.cas.cz}
\date{}

\begin{abstract}
We prove that separable extensions of noetherian rings and finite \'etale morphisms of noetherian schemes give rise to separable extensions of singularity categories.
\end{abstract}

\maketitle

\section{Introduction}\label{sec:introduction}

The theory of separable extensions in tensor-triangular geometry has received great interest in the past couple of decades. Initiated by Balmer~\cite{Balmer11,Balmer14} as a novel technique of constructing new triangulated categories and demonstrated in commutative algebra, it subsequently expanded to several equivariant contexts~\cite{BalmerDellAmbrogioSanders15}, to modular representation theory~\cite{Balmer15,BalmerCarlson18} and to algebraic geometry~\cite{Balmer16}. The general idea was to uniformly capture naturally arising constructions that were beyond the reach of the pre-existing theory (of localization for instance) by expressing them as categories of modules over certain monads by establishing monadicity of appropriate adjunctions. These include separable extensions of commutative rings, separated \'etale morphisms of quasi-compact quasi-separated schemes and restriction to subgroups, among others. More recent developments have also emerged in the theory of $\infty$-categories~\cite{NaumannPol24,Ramzi23}.

We expand the above list by incorporating two results concerning singularity categories (see~\Cref{rec:sing-recollement}) in analogy with the corresponding results~\cite[Theorem 6.5]{Balmer11} and~\cite[Theorem 3.5]{Balmer16} for derived categories.

\begin{thm*}[{\ref{thm:singularity-separable}}]
Let $R$ be a commutative noetherian ring and let $A$ be a separable flat $R$-algebra that is a noetherian ring (not necessarily finitely generated over $R$). There is an equivalence $E_S\colon \rmS(A)\xr{\simeq} \Mod_{\rmS(R)}A$ of triangulated categories between the singularity category of $A$ and the category of modules in $\rmS(R)$ over the monad $A\ot_R-\colon \rmS(R)\to \rmS(R)$.
\end{thm*}

\begin{thm*}[{\Cref{thm:singularity-etale-schemes}, \Cref{cor:small-singularity-etale-schemes}}]
Let $f\colon U\to X$ be a finite \'etale morphism of noetherian schemes. There is an equivalence $E\colon \rmS(U)\xr{\simeq} \Mod_{\rmS(X)}\rmR f_\ast \mcO_U$ of triangulated categories between the singularity category of $U$ and the category of modules in $\rmS(X)$ over the monad $\rmR f_\ast \mcO_U \ot_X -\colon \rmS(X)\to \rmS(X)$. The equivalence $E$ yields an equivalence $E_{\mathrm{sg}}\colon \rmD_{\mathrm{sg}}(U)^\natural\xr{\simeq} \Mod_{\rmD_{\mathrm{sg}}(X)^\natural} \rmR f_\ast \mcO_U$.
\end{thm*}

\begin{Ack}
The author thanks Michal Hrbek and Greg Stevenson for their comments. This research was supported by the project LQ100192601 Lumina quaeruntur, funded by the Czech Academy of Sciences (RVO 67985840).
\end{Ack}

\section{Monads and modules}\label{sec:monads-and-modules}

We recall some definitions and results involving monads and adjunctions and refer the reader to the articles~\cite{Balmer11,Balmer16} and references therein for more details. For the basic theory of monads, we recommend~\cite{Riehl16}. Let $\scC$ be a triangulated category.

\begin{defn}
\label[defn]{defn:monad}
A \emph{triangulated monad} on $\scC$ is a triple $M=(M,\mu,\eta)$ that consists of a triangulated functor $M\colon \scC\to \scC$ and natural transformations $\mu\colon M^2 \to M$ (called \emph{multiplication}) and $\eta\colon \Id_\scC\to M$ (called the \emph{unit}) that commute with suspension such that the following diagrams commute:
\[
\begin{tikzcd}[row sep=3em]
M^3 \rar["M\mu"] \dar["\mu_M"'] & M^2 \dar["\mu"]
\\
M^2 \rar["\mu"] & M
\end{tikzcd}
\qquad \& \qquad
\begin{tikzcd}[row sep=3em]
M \rar["M\eta"] \drar[equal] & M^2 \dar["\mu"description] & \lar["\eta_M"'] M \dlar[equal]
\\
& M.
\end{tikzcd}
\]
\end{defn}

\begin{defn}
\label[defn]{defn:module}
An $M$-\emph{module} is a pair $(X,\alpha)$ that consists of an object $X\in \scC$ and a morphism $\alpha \colon M(X)\to X$ (called the \emph{action} of $M$ on $X$) such that the following diagrams commute:
\[
\begin{tikzcd}[row sep=3em]
M^2(X) \rar["M\alpha"]  \dar["\mu_X"'] & M(X) \dar["\alpha"]
\\
M(X) \rar["\alpha"] & X
\end{tikzcd}
\qquad \& \qquad
\begin{tikzcd}[row sep=3em]
X \rar["\eta_X"] \drar[equal] & M(X) \dar["\alpha"]
\\
& X.
\end{tikzcd}
\]
A \emph{morphism} of $M$-modules $f\colon (X,\alpha)\to (Y,\beta)$ is a morphism $f\colon X\to Y$ that commutes with the actions, meaning that the square below commutes:
\[
\begin{tikzcd}[row sep=3em]
M(X) \rar["\alpha"] \dar["Mf"']& X \dar["f"]
\\
M(Y) \rar["\beta"] & Y.
\end{tikzcd}
\]
\end{defn}

The category of $M$-modules (also known as the category of $M$-algebras or the Eilenberg--Moore category) is denoted by $\Mod_\scC M$ and it is an additive category that inherits a suspension functor $\gS\colon \Mod_\scC M \to \Mod_\scC M$ in the obvious way. In general, $\Mod_\scC M$ does not inherit the structure of a triangulated category from that of $\scC$; it does when $M$ is separable and $\scC$ is idempotent complete. Before we discuss this, let us recall some basic facts about adjunctions and monads.

\begin{rec}
\label[rec]{rec:adj-real}
Given an adjunction $\begin{tikzcd}[cramped,column sep=1em] F\colon\scC \rar[shift left] & \scD \rcolon G \lar[shift left]\end{tikzcd}$ with unit $\eta\colon \Id_\scC \to GF$ and counit $\varepsilon\colon FG\to \Id_\scD$, the triple $(GF,G\varepsilon F,\eta)$ is a monad on $\scC$. Given a monad $(M,\mu,\eta)$ on $\scC$, we say that the adjunction $F\dashv G$ \emph{realizes} $M$ if $GF=M$, the unit of $F\dashv G$ coincides with that of $M$ and $G\varepsilon F=\mu$.
\end{rec}

\begin{rec}
\label[rec]{rec:free-forgetful-adj}
Given an object $X\in \scC$, there is an associated \emph{free} $M$-module defined by $(M(X),\mu_X)$. This assignment defines a functor $F_M\colon \scC\to \Mod_\scC M$ with right adjoint $U_M\colon \Mod_\scC M\to \scC$ the \emph{forgetful functor} that sends an $M$-module $(X,\alpha)$ to $X\in \scC$. Clearly, $U_MF_M=M$. The unit of the adjunction $F_M\dashv U_M$ is the unit $\eta$ of the monad $M$ and the counit $\varepsilon_M\colon F_MU_M\to \Id_{\Mod_\scC M}$ has components ${\varepsilon_M}_{(X,\alpha)}=\alpha\colon (M(X),\mu_X)\to (X,\alpha)$. The functors $F_M$ and $U_M$ are additive and commute with suspension. Further, $U_M\varepsilon_M F_M=\mu$, so the adjunction
\[
\begin{tikzcd}
\scC \rar[shift left=5pt,"F_M"] \rar[phantom,"\perp"]& \lar[shift left=5pt,"U_M"] \Mod_\scC M,
\end{tikzcd}
\]
called the \emph{Eilenberg--Moore adjunction} (or the \emph{free-forgetful adjunction}) realizes $M$ and satisfies the following universal property: If
\[
\begin{tikzcd}
\scC \rar[shift left=5pt,"F"] \rar[phantom,"\perp"]& \lar[shift left=5pt,"G"] \scD,
\end{tikzcd}
\quad
\eta\colon \Id_\scC \to GF,
\quad
\varepsilon\colon FG\to \Id_\scD
\]
is an adjunction that realizes $M$, then there is a unique functor $E\colon \scD \to \Mod_\scC M$ satisfying $EF=F_M$ and $U_ME=G$. The functor $E$ sends an object $Y\in \scD$ to the $M$-module $(G(Y),MG(Y)=GFG(Y)\xr{G\varepsilon_Y} G(Y))$ and we refer to $E$ as the \emph{comparison functor}. If the functors $F$ and $G$ are additive and commute with suspension, then $E$ is also an additive functor that commutes with suspension. The image of $F_M$ in $\Mod_\scC M$ forms the subcategory $\Free_\scC M$ of free $M$-modules (also known as the Kleisli category). The free-forgetful adjunction between $\scC$ and $\Free_\scC M$ satisfies an analogous universal property: it is initial among the adjunctions realizing $M$; see the references given at the beginning of the section for details.
\end{rec}

Next we discuss separability of a monad and some of its consequences. Separability is a key property under which the associated category of modules exhibits compatibility with respect to the triangulated structure and, as we will see in the sequel, the tensor structure.

\begin{defn}
\label[defn]{defn:separable-monad}
A triangulated monad $M\colon \scC\to \scC$ is called \emph{separable} if the multiplication $\mu\colon M^2 \to M$ admits a section $\sigma \colon M\to M^2$, meaning that $\mu \circ \sigma=\Id_{M}$ (that we assume commutes with suspension) and $M\mu \circ \sigma_M=\sigma \circ \mu=\mu_M\circ M\sigma$.
\end{defn}

\begin{thm}[{\cite[Main Theorem 5.17]{Balmer11}}]
\label[thm]{thm:modules-triangulated}
With notation as in~\Cref{rec:free-forgetful-adj}, when $M$ is separable and $\scC$ is idempotent complete, the category $\Mod_\scC M$ inherits a triangulated structure such that the functors $F_M$ and $U_M$ are triangulated functors. The triangles of $\Mod_\scC M$ are precisely those triangles of modules whose underlying diagram is a triangle in $\scC$. Further, there is an equivalence $(\Free_\scC M)^\natural\simeq \Mod_\scC M$. If $F\dashv G$ is an adjunction of triangulated functors that realizes $M$, then the comparison functor $E\colon \scD\to \Mod_\scC M$ is a triangulated functor. If, moreover, $\scD$ is idempotent complete and the counit of the $F\dashv G$ adjunction has a section, then $E$ is an equivalence of triangulated categories.
\end{thm}

The last part of~\Cref{thm:modules-triangulated}, regarding when the comparison functor is an equivalence, has undergone a very useful improvement: It is enough to know that the right adjoint $G$ is conservative. The precise statement goes as follows.

\begin{thm}[{\cite{DellAmbrogioSanders18}}]
\label[thm]{rem:comparison-equivalence-no-sep}
Let $\scC$ and $\scD$ be idempotent complete triangulated categories and $F\dashv G$ an adjunction of triangulated functors that realizes a triangulated monad $M\colon \scC\to \scC$. If $\Mod_\scC M$ admits a triangulated structure that is compatible with $F_M$ and $U_M$ and $G$ is conservative, then the comparison functor $E\colon \scD\to \Mod_\scC M$ is an equivalence, which does not necessarily preserve the triangulated structures. If the triangles of $\Mod_\scC M$ are precisely those whose underlying diagram is a triangle in $\scC$, then $E$ is a triangulated functor.
\end{thm}

The final result we need to recall is Balmer's ``separable Neeman--Thomason theorem".

\begin{thm}[{\cite[Theorem 4.2]{Balmer16}}]
\label[thm]{thm:modules-cg}
If $\scC$ is a compactly generated triangulated category and $M\colon \scC \to \scC$ is a separable triangulated monad that preserves coproducts, then $\Mod_\scC M$ is also compactly generated. Its subcategory of compact objects is $(\Mod_\scC M)^\rmc=\thick(F_M(\scC^\rmc))$ (note that since $U_M$ preserves coproducts, $F_M$ preserves compact objects and thus, the preceding equality makes sense). If $M$ preserves compact objects, then $(\Mod_\scC M)^\rmc=\Mod_{\scC^\rmc}M^\rmc$, where $M^\rmc\colon \scC^\rmc \to \scC^\rmc$ is the restriction of $M$ on $\scC^\rmc$.
\end{thm}

For ease of reference, we summarize the three aforementioned theorems.

\begin{summ}
\label[summ]{summ:summary}
Suppose $M\colon \scC\to \scC$ is a separable triangulated monad realized by an adjunction $\begin{tikzcd}[cramped,column sep=1em] F\colon\scC \rar[shift left] & \scD \rcolon G \lar[shift left]\end{tikzcd}$ of triangulated functors with $\scC$ and $\scD$ being idempotent complete. If $G$ is conservative, then the comparison functor $E\colon \scD \to \Mod_\scC M$ is an equivalence of triangulated categories. If $\scC$ is compactly generated, then $\Mod_\scC M$ is also compactly generated.
\end{summ}

\section{Ring objects}\label{sec:ring-objects}

Let $\scT=(\scT,\ot,1)$ be a tensor-triangulated category, i.e., a triangulated category equipped with a symmetric tensor product that is a triangulated functor in both variables. As in~\cite{BalmerFavi11}, we say that $\scT$ is \emph{rigidly-compactly generated} if $\scT$ is compactly generated, the tensor product preserves coproducts in both variables, the subcategory $\scT^\rmc$ of compact objects is a tensor-triangulated subcategory, so $1\in \scT^\rmc$ and $\scT^\rmc \ot \scT^\rmc \subseteq \scT^\rmc$, that coincides with the subcategory of rigid objects. We will not assume that $\scT$ is rigidly-compactly generated, unless stated explicitly.

A wide range of monads on a tensor-triangulated category is provided by ring objects, referred to as monoids in most classical sources, e.g.,~\cite{MacLane98}.

\begin{defn}
\label[defn]{defn:ring-object}
A \emph{ring object} in $\scT$ is a triple $A=(A,\mu,\eta)$ that consists of an object $A\in \scT$ and morphisms $\mu\colon A\ot A\to A$ and $\eta\colon 1\to A$ such that the triple $(A\ot -,\mu\ot -,\eta\ot -)$ is a monad. A ring object $A$ is called:
\begin{enumerate}[\rm(i)]
\item
\emph{commutative} if $\mu\circ (1\,2)=\mu$, where $(1\,2)$ is the swap of factors.
\item
\emph{separable} if $\mu\colon A\ot A\to A$ has a section $\sigma\colon A\to A\ot A$, with respect to which, $(A\ot-,\mu\ot-,\eta\ot-)$ is a separable monad.
\end{enumerate}
\end{defn}

\begin{rem}
\label[rem]{rem:sep-monad-not-sep-ring}
At this point, an important distinction needs to be made. In the definition of separable ring object, we require the multiplication map $\mu\colon A\ot A\to A$ to admit a section $\sigma\colon A\to A\ot A$ that renders the corresponding monad separable. If the monad $A\ot-$ is separable with section $\xi\colon A\ot-\to A\ot A\ot -$, then the morphism $\xi_1\colon A\to A\ot A$ is a section for $\mu\colon A\ot A\to A$ but it may not be the case that $\xi_A=\xi_1 \ot A$ and so, $A$ may not be a separable ring object (at least not with respect to $\xi_1$); see~\cite[Remark 2.10]{BalmerDellAmbrogioSanders15} for further discussion.
\end{rem}

\begin{rem}
\label[rem]{rem:modules-tensor}
If $A\in \scT$ is a commutative separable ring object and $\scT$ is idempotent complete, then $\Mod_\scT A$ admits a tensor structure that renders it a tensor-triangulated category such that the free module functor $F_A\colon \scT \to \Mod_\scT A$ is a tensor-triangulated functor~\cite{Balmer14}. We will discuss this in more detail in~\Cref{cons:action-of-module-cats}. When $\scT$ is rigidly-compactly generated, $\Mod_\scT A$ is also rigidly-compactly generated.
\end{rem}

\begin{ex}
\label[ex]{ex:ring-objects}
The following are some examples of ring objects.
\begin{enumerate}[\rm(a)]
\item
Let $F\colon \scT_1 \to \scT_2$ be a coproduct-preserving tt-functor between rigidly-compactly generated tt-categories and let $G\colon \scT_2\to \scT_1$ be its right adjoint and denote by $\varepsilon\colon FG\to \Id_{\scT_2}$ the counit of adjunction. By the projection formula~\cite{BalmerDellAmbrogioSanders16}
\[
G(F(X) \ot Y)\cong X\ot G(Y),\, \forall X\in \scT_1,\, \forall Y\in \scT_2,
\]
it follows that $GF\cong G(1)\ot -$. Then $G(1)$ is a commutative ring object with multiplication the natural map
\[
G(1)\ot G(1)\cong GFG(1)\xr{G\varepsilon_{1}} G(1).
\]
\item
If $A$ is a right idempotent in a rigidly-compactly generated tt-category $\scT$, meaning that there is a morphism $\eta\colon 1\to A$ such that $\sigma\coloneqq \eta \ot A\colon A\to A\ot A$ is an isomorphism, then $(A,\sigma^{-1},\eta)$ is a commutative separable ring object and $\Mod_\scT A=\scT/\Ker(A\ot -)$. By~\cite{BalmerFavi11}, every right idempotent $A$ occurs as the image $L_A(1)\in \scT$ for a tensor smashing localization functor $L_A\colon \scT\to \scT$. Hence, tensor smashing localizations are exactly the categories of modules over those ring objects whose multiplication is an isomorphism.
\item
Let $R$ be a commutative ring and let $A$ be a separable flat $R$-algebra. Then $A$ is a separable ring object in $\rmD(R)$. Moreover, $\Mod_{\rmD(R)}A$ is equivalent to $\rmD(A)$; see~\cite[Theorem 6.5]{Balmer11}.
\item
Let $X$ be a quasi-compact quasi-separated scheme. If $f\colon U\to X$ is a separated \'etale morphism, then $\rmR f_\ast \mcO_U$ is a commutative separable ring object in $\rmD_\mathrm{qc}(X)$ and it holds that $\Mod_{\rmD_\mathrm{qc}(X)}\rmR f_\ast \mcO_U\simeq \rmD_\mathrm{qc}(U)$; see~\cite[Theorem 3.5]{Balmer16}. If $X$ is a noetherian scheme, then all commutative separable ring objects in $\rmD_\mathrm{qc}(X)$ are of the form $\rmR f_\ast L \mcO_U$, for some \'etale morphism $f\colon U\to X$ and a tensor smashing localization functor $L\colon \rmD_\mathrm{qc}(U)\to \rmD_\mathrm{qc}(U)$; see~\cite[Theorem 7.10]{Neeman18}. The tensor smashing localizations of $\rmD_\mathrm{qc}(U)$ are all finite and correspond to specialization closed subsets $V\subseteq U$; see~\cite{AlonsoTarrioJeremiasLopezSoutoSalorio04,BalmerFavi11}. The situation is depicted in the following diagram:
\[
\begin{tikzcd}[column sep=2.58em, row sep=5em]
\arrow[dd,bend right,shift right=35pt,shorten <=-10pt] \arrow[dd,phantom,bend right,shift right=30pt,"\dashv"] \rmD_\mathrm{qc}(X) \dar[phantom,"\dashv"] \rar[phantom,"\perp"] \rar[shift left=5pt,"f^\ast"] \dar[shift right=5pt]
& \rmD_\mathrm{qc}(U) \dar[phantom,"\dashv"] \rar[phantom,"\perp"] \lar[shift left=5pt,"\rmR f_\ast"] \arrow[r,shift left=5pt,two heads] \dar[shift right=5pt] \dlar["\simeq"]
& \arrow[l,shift left=5pt,hook'] \rmD_\mathrm{qc}(U)/\Ker(L\mcO_U\ot^\rmL_U -) \dlar["\simeq"]
\\
\uar[shift right=5pt]\Mod_{\rmD_\mathrm{qc}(X)}\rmR f_\ast \mcO_U \dar[shift right=5pt,two heads] \dar[phantom,"\dashv"]
& \Mod_{\rmD_\mathrm{qc}(U)}L\mcO_U \dlar["\simeq"] \uar[shift right=5pt]
\\
\arrow[uu,bend left,shift left=25pt,shorten >=-10pt] \uar[shift right=5pt,hook] \Mod_{\rmD_\mathrm{qc}(X)}\rmR f_\ast L\mcO_U.
\end{tikzcd}
\]
\item
Let $G$ be a cyclic $p$-group and let $k$ be a separably closed field with $\Char(k)=p$. Then all finite-dimensional commutative separable ring objects in the stable category $\sMod kG$ are of the form $kX$, where $X$ is a finite $G$-set; see~\cite{BalmerCarlson18}.
\end{enumerate}
\end{ex}

\section{Monads via tensor-actions}\label{sec:monads-via-tensor-actions}

From now on, we will tacitly assume that all categories involved are idempotent complete. Let $\scT$ be a tensor-triangulated category and let $\scK$ be a triangulated category. An \emph{action} of $\scT$ on $\scK$ is a functor $\ast \colon \scT \times \scK \to \scK$ that is triangulated in each variable such that $(X\ot Y) \ast A\cong X \ast (Y \ast A)$ and $1\ast A\cong A$, $\forall X,Y\in \scT,\, \forall A\in \scK$; see~\cite{Stevenson13} for more details. When $\scT$ is rigidly-compactly generated and $\scK$ is compactly generated, we require the action to preserve coproducts in each variable.

One can readily verify that a separable ring object $A\in \scT$ gives rise to a separable triangulated monad $A\ast -\colon\scK\to \scK$ with multiplication, unit and section of multiplication given by acting with $\mu$, $\eta$ and $\sigma$, respectively. We denote by $\Mod_\scK A$ the category of modules over the monad $A\ast-\colon \scK\to \scK$. Applying the theorems presented in~\Cref{sec:monads-and-modules} in this setting, we have that $\Mod_\scK A$ is a triangulated category and if the monad $A\ast -\colon \scK\to \scK$ is realized by an adjunction of triangulated functors $\begin{tikzcd}[cramped,column sep=1em] F\colon\scK \rar[shift left] & \scD \lar[shift left] \rcolon G\end{tikzcd}$ such that $G$ is conservative, then the comparison functor $\scD \to \Mod_\scK A$ is an equivalence of triangulated categories. In case $\scT$ is rigidly-compactly generated and $\scK$ is compactly generated, the monad $A\ast -\colon \scK\to \scK$ preserves coproducts and $\Mod_\scK A$ is compactly generated; see~\Cref{summ:summary}.

\begin{cons}
\label[cons]{cons:action-of-module-cats}
If $A\in \scT$ is a commutative and separable ring object, then there is an action of $\Mod_\scT A$ on $\Mod_\scK A$ defined as follows. Let $(M,\alpha)\in \Mod_\scT A$ and $(N,\beta)\in \Mod_\scK A$. Then the morphism
\[
e\coloneqq M\ast N \cong M\ast 1 \ast N \xr{\eta} M\ast A \ast N \xr{\sigma} M\ast A\ast A\ast N \xr{\alpha\ast \beta}M\ast N
\]
is an idempotent, which therefore splits. We define $M\ast_A N=\im e$. The following diagram is a split coequalizer:
\[
\begin{tikzcd}
M\ast A \ast N \rar[shift left,"\alpha\ast N"] \rar[shift right,"M\ast \beta"'] & M\ast N \rar[two heads,"e"] & M \ast_A N.
\end{tikzcd}
\]
The $A$-module structure on $M\ast_A N$ is given by the unique map that makes the following diagram commute:
\[
\begin{tikzcd}[column sep=7em,row sep=4em]
M\ast A \ast A \ast N \rar["A\ast(\alpha\ast N - M\ast \beta)"] \dar["M\ast \mu \ast N"']& A\ast M\ast N \rar[two heads,"A \ast e"] \dar["\alpha\ast N - M\ast  \beta"] & \dar["\exists!"',dashed] A \ast (M \ast_A N)
\\
M\ast A \ast N \rar["\alpha\ast N - M\ast \beta"] & M\ast N \rar[two heads,"e"] & M \ast_A N.
\end{tikzcd}
\]
For $\scK=\scT$ and $\ast=\ot$, this is exactly how the tensor product on $\Mod_\scT A$ is defined in~\cite{Balmer14}.
This construction defines an action of $\Mod_\scT A$ on $\Mod_\scK A$ and the functors $F_\scT\colon \scT\to \Mod_\scT A$ and $F_\scK\colon \scK\to \Mod_\scK A$ satisfy the relation $F_\scK(X\ast E)=F_\scT(X)\ast_A F_\scK(E),\, \forall X\in \scT,\, \forall E\in \scK$. Alternatively, one can first define the action of $\Free_\scT A$ on $\Free_\scK A$ using this relation and then extend this action on the categories $\Mod_\scT A$ and $\Mod_\scK A$ via idempotent completion.
\end{cons}

\section{Separable algebras}\label{sec:separable-algebras}

In this section, we provide some relevant background on derived categories and singularity categories and proceed to prove~\Cref{thm:singularity-separable}, concerning singularity categories of separable algebras (cf.~\cite[Theorem 6.5]{Balmer11}).

Let $X$ be a noetherian scheme. It is well known that the category $\QCoh X$ of quasi-coherent sheaves is a locally noetherian Grothendieck abelian category and that direct sums of injective quasi-coherent sheaves are injective. We denote by $\rmD(X)=\rmD(\QCoh X)$ its derived category. It holds that $\rmD(X)$ is equivalent to the category $\rmD_{\mathrm{qc}}(\Mod \mcO_X)$ of complexes of sheaves with quasi-coherent cohomology; see~\cite[Tag 09T4]{Stacks26} and also the relevant discussion within the proof of~\cite[Theorem 7.13]{Positselski25}. Further, $\rmD(X)$ is a rigidly-compactly generated tensor-triangulated category and its subcategory of compact objects is $\rmD(X)^\rmc=\rmD^{\mathrm{perf}}(X)$ the subcategory of perfect complexes; see e.g.,~\cite[Examples 1.2]{BalmerFavi11}. Let $\rmK(X)=\rmK(\QCoh X)$ be the homotopy category of complexes of quasi-coherent sheaves and let $\rmK(\Inj X)$ denote its subcategory of complexes of injective quasi-coherent sheaves. We call the category $\rmS(X)=\rmK_{\mathrm{ac}}(\Inj X)$, consisting of acyclic complexes of injective quasi-coherent sheaves, the \emph{singularity category} of $X$. 

\begin{rec}
\label[rec]{rec:sing-recollement}
By the work of~\cite{Krause05}, the categories $\rmK(\Inj X)$ and $\rmS(X)$ are compactly generated and are related with $\rmD(X)$ via a recollement
\begin{equation}
\label{eq:recollement}    
\begin{tikzcd}
[column sep=5em]
\rmS(X) \rar[phantom,"\perp",shift left=1.8ex] \rar[hook] \rar[phantom,"\perp",shift right=1.8ex]
& \rmK(\Inj X) \lar["I_\lambda"',shift right=3.5ex] \lar["I_\rho",shift left=3.5ex] \rar[phantom,"\perp",shift left=1.8ex] \rar \rar[phantom,"\perp",shift right=1.8ex]
& \rmD(X) \lar["Q_\lambda"',shift right=3.5ex] \lar["Q_\rho",shift left=3.5ex]
\end{tikzcd}
\end{equation}
where the functor $\rmK(\Inj X)\to \rmD(X)$ is the composite $\rmK(\Inj X)\hookrightarrow \rmK(X)\twoheadrightarrow \rmD(X)$. The \emph{stabilization functor} $I_\lambda Q_\rho\colon \rmD(X)\to \rmS(X)$ induces an equivalence of triangulated categories $\rmD_{\mathrm{sg}}(X)^\natural\xr{\simeq} \rmS(X)^\rmc$ between the idempotent completion of the \emph{small singularity category} $\rmD_{\mathrm{sg}}(X)=\rmD^\rmb(\Coh X)/\rmD^{\mathrm{perf}}(X)$ and the subcategory of compact objects of $\rmS(X)$. These results were first established for noetherian separated schemes. Separatedness was used to ensure the equivalence $\rmD(X)\simeq \rmD_{\mathrm{qc}}(\Mod \mcO_X)$. As mentioned in the preceding discussion, separatedness turns out unnecessary. 
\end{rec}

\begin{rem}
\label[rem]{rem:abelian-sing-cats}
In~\cite{Krause05}, the recollement~\eqref{eq:recollement} is studied in the context of a locally noetherian Grothendieck abelian category $\scA$, subject to the condition that $\rmD(\scA)$ is compactly generated. In particular, the analogous results hold if we replace the scheme $X$ by a noetherian, not necessarily commutative, ring $A$.
\end{rem}

Let $E\in \rmD(X)$ and let $\wt{E}$ be a $\rmK$-flat resolution of $E$, i.e., $\wt{E}\ot_X -$ preserves acyclic complexes and $\wt{E}$ is quasi-isomorphic to $E$. By~\cite[Corollary 3.22]{Murfet07}, we can assume that $\wt{E}$ consists of flat quasi-coherent sheaves. The functor $\ast\colon\rmD(X)\times \rmS(X)\to \rmS(X)$ defined by $E\ast D=\wt{E}\ot_X D$ is an action of $\rmD(X)$ on $\rmS(X)$; see~\cite{Stevenson14}.

In the case where $X=\Spec R$ is a noetherian affine scheme and $A$ is a separable flat $R$-algebra, it holds that $A$ is a separable ring object in $\rmD(R)$ and the action $A\ot_R-\colon \rmS(R)\to \rmS(R)$ is a separable coproduct-preserving triangulated monad on $\rmS(R)$. According to~\Cref{summ:summary}, $\Mod_{\rmS(R)}A$ is a compactly generated triangulated category and if $\begin{tikzcd}[cramped,column sep=1em] F\colon\rmS(R)\rar[shift left] & \lar[shift left] \scD \rcolon G\end{tikzcd}$ is an adjunction of triangulated functors that realizes the monad $A\ot_R-\colon \rmS(R)\to \rmS(R)$ and $G$ is conservative, e.g., faithful, then the comparison functor $\scD\to \Mod_{\rmS(R)}A$ is an equivalence of triangulated categories.

Our theorem below identifies $\Mod_{\rmS(R)}A$ with $\rmS(A)$, when $A$ is a noetherian ring. The proof relies on a decomposition of a certain adjunction, so we first describe the relevant situation more generally for clarity.

\begin{rec}
\label[rec]{rec:unit-counit-of-composite}
Consider a sequence of adjunctions
\[
\begin{tikzcd}
\scC_1 \rar[shift left=5pt,"F_1"] \rar[phantom,"\perp"] & \scC_2 \rar[shift left=5pt,"F_2"] \rar[phantom,"\perp"] \lar[shift left=5pt,"G_1"] & \scC_3 \rar[shift left=5pt,"F_3"] \rar[phantom,"\perp"] \lar[shift left=5pt,"G_2"] & \scC_4 \lar[shift left=5pt,"G_3"]
\end{tikzcd}
\]
and let $\eta^i\colon \Id_{\scC_i} \to G_iF_i$ and $\varepsilon^i\colon F_iG_i\to \Id_{\scC_{i+1}}$ be the unit and the counit of the $F_i\dashv G_i$ adjunction, for $i=1,2,3$. Then the unit $\eta\colon \Id_{\scC_1}\to G_1G_2G_3F_3F_2F_1$ and the counit $\varepsilon\colon F_3F_2F_1G_1G_2G_3\to \Id_{\scC_4}$ of the $F_3F_2F_1\dashv G_1G_2G_3$ adjunction are expressed via the composites
\begin{align*}
\eta&=\Id_{\scC_1} \xr{\eta^1} G_1F_1 \xr{G_1(\eta^2_{F_1(-)})} G_1G_2F_2F_1 \xr{G_1G_2(\eta^3_{F_2F_1(-)})} G_1G_2G_3F_3F_2F_1,\\
\varepsilon&=F_3F_2F_1G_1G_2G_3 \xr{F_3F_2(\varepsilon^1_{G_2G_3(-)})} F_3F_2G_2G_3 \xr{F_3(\varepsilon^2_{G_3(-)})} F_3G_3 \xr{\varepsilon^3} \Id_{\scC_4}.
\end{align*}
\end{rec}

\begin{thm}
\label[thm]{thm:singularity-separable}
Let $R$ be a commutative noetherian ring and let $A$ be a separable flat $R$-algebra that is a noetherian ring (not necessarily finitely generated over $R$). There is an equivalence $E_S\colon \rmS(A)\xr{\simeq} \Mod_{\rmS(R)}A$ of triangulated categories between the singularity category of $A$ and the category of modules in $\rmS(R)$ over the monad $A\ot_R-\colon \rmS(R)\to \rmS(R)$.
\end{thm}

\begin{proof}
Since $A$ is flat, $A\ot_R- \colon \Mod R\to \Mod A$ is exact. Viewing $A$ as an object in $\rmD(R)$, its action on $\rmS(R)$ is thus given by $A\ot_R-\colon \rmS(R)\to \rmS(R)$. The functor $\res\colon \Mod A\to \Mod R$ is exact and preserves injectives. This yields the adjunction
\begin{equation}
\label{eq:singular-adj}
\begin{tikzcd}[column sep=5em]
\rmS(R)\rar["\lambda_A A\ot_R-",shift left=5pt] \rar[phantom,"\perp"] & \lar["\res",shift left=5pt] \rmS(A),
\end{tikzcd}
\end{equation}
where $\lambda_A\colon \rmK(A)\to \rmS(A)$ is left adjoint to the inclusion $\rmS(A)\hookrightarrow \rmK(A)$. We claim that adjunction~\eqref{eq:singular-adj} realizes the monad $A\ot_R-\colon \rmS(R)\to \rmS(R)$. Consider the square
\begin{equation}
\label{eq:lambda-res-commute}
\begin{tikzcd}[row sep=2.8em]
\rmK(A) \rar["\lambda_A"] \dar["\res"']& \rmS(A) \dar["\res"]
\\
\rmK(R) \rar["\lambda_R"] & \rmS(R).
\end{tikzcd}
\end{equation}
Since $A$ is flat, the functor $\Hom_R(A,-)\colon \Mod R\to \Mod A$ applied to an acyclic complex of injectives yields an acyclic complex of injectives; see~\cite[Proposition 3.1]{Emmanouil23} and~\cite[Proposition 2.8]{Verasdanis25}. Taking right adjoints, the square~\eqref{eq:lambda-res-commute} transposes to the commutative square
\[
\begin{tikzcd}[row sep=2.8em]
\rmK(A) & \rmS(A) \lar[hook']
\\
\rmK(R) \uar["\Hom_R(A{,}-)"]& \uar["\Hom_R(A{,}-)"'] \lar[hook'] \rmS(R).
\end{tikzcd}
\]
We conclude that square~\eqref{eq:lambda-res-commute} commutes, i.e., $\res\lambda_A \cong \lambda_R \res$. Hence, we have
\[
\res \lambda_A (A\ot_R E)\cong\lambda_R \res (A\ot_R E)=\lambda_R (A\ot_R E)\cong A\ot_R E,\, \forall E\in \rmS(R),
\]
with the last isomorphism because $A\ot_R E\in \rmS(R)$. So, $\res \lambda_A A\ot_R -\cong A\ot_R -$. Adjunction~\eqref{eq:singular-adj} can be decomposed as
\begin{equation}
\label{eq:decomp}
\begin{tikzcd}[column sep=3.5em]
\rmS(R)\rar[shift left=5pt,hook] \rar[phantom,"\perp"] & \rmK(R) \rar["A\ot_R-",shift left=5pt] \lar["\rho_R",shift left=5pt]\rar[phantom,"\perp"] & \lar["\res",shift left=5pt] \rmK(A) \rar[phantom,"\perp"] \rar["\lambda_A",shift left=5pt] & \rmS(A). \lar[shift left=5pt,hook']
\end{tikzcd}
\end{equation}
Let $\eta^\rho,\eta,\eta^\lambda,\varepsilon^{\rho},\varepsilon,\varepsilon^{\lambda}$ be the units and the counits of these adjunctions, respectively from left to right. Let $\ol{\eta}$ and $\ol{\varepsilon}$ be the unit and the counit of the adjunction~\eqref{eq:singular-adj}. Then it holds that $\ol{\eta}_E$ coincides with the unit map $E\to A\ot_R E$ and $\res \ol{\varepsilon}_{\lambda_A(A\ot_R E)}$ coincides with the multiplication map $A\ot_R A \ot_R E\to A\ot_R E$, for all $E\in \rmS(R)$. This is due to a combination of the following arguments:
\begin{enumerate}
\item
The unit $\ol{\eta}$ and the counit $\ol{\varepsilon}$ are expressed as composites involving $\eta^\rho,\eta,\eta^\lambda$ and $\varepsilon^{\rho},\varepsilon,\varepsilon^{\lambda}$, respectively, as described in~\Cref{rec:unit-counit-of-composite}.
\item
Since $E\in \rmS(R)$, it holds that $\eta^\rho_E$ and $\varepsilon^{\rho}_{\res \lambda_A(A\ot_R E)}$ are isomorphisms.
\item
The adjunction $A\ot_R \dashv \res$, as in the middle of~\eqref{eq:decomp}, realizes the monad $A\ot_R-\colon \rmK(R)\to \rmK(R)$.
\item
$\res \lambda_A A\ot_R-\cong A\ot_R -$, as discussed previously.
\end{enumerate}
We infer that adjunction~\eqref{eq:singular-adj} realizes the monad $A\ot_R-\colon \rmS(R)\to \rmS(R)$. Since $A$ is separable, the counit of the $A\ot_R-\dashv \res$ adjunction between module categories has a section. Since the counit of the $A\ot_R- \dashv \res$ adjunction between homotopy categories of complexes is defined degree-wise, it also has a section. In particular, this implies that the functor $\res\colon \rmK(A)\to \rmK(R)$ is faithful, which of course remains faithful when restricted to singularity categories. In conclusion, adjunction~\eqref{eq:singular-adj} realizes the monad $A\ot_R-\colon \rmS(R)\to \rmS(R)$ and the right adjoint $\res\colon \rmS(A)\to \rmS(R)$ is faithful. By the discussion preceding~\Cref{rec:unit-counit-of-composite}, the comparison functor $E_S\colon \rmS(A)\to \Mod_{\rmS(R)}A$ is an equivalence of triangulated categories.
\end{proof}

\begin{rem}
\label[rem]{rem:singularity-etale-algebra}
The proof of~\Cref{thm:singularity-separable} is noticeably more straightforward if one assumes that $A$ is an \'etale $R$-algebra, i.e., separable and finitely generated (hence necessarily projective) over $R$. In this case, the functors $A\ot_R-\colon \Mod R\to \Mod A$ and $\Hom_R(A,-)\colon \Mod R\to \Mod A$ are naturally isomorphic. Therefore, $A\ot -$ preserves injectives and the left adjoint in adjunction~\eqref{eq:singular-adj} is already $A\ot_R-$.
\end{rem}

Next we discuss the compatibility of the equivalences $E_{\rmD}\colon \rmD(A)\xr{\simeq} \Mod_{\rmD(R)}A$ and $E_{\rmS}\colon \rmS(A)\xr{\simeq} \Mod_{\rmS(R)}A$ with the various tensor-actions involved, when $A$ is commutative.

\begin{prop}
\label[prop]{prop:compatibility-of-actions}
Let $R$ be a commutative noetherian ring and let $A$ be a separable flat $R$-algebra that is a commutative noetherian ring. Let $F_\rmD\colon \rmD(R)\to \Mod_{\rmD(R)}A$ and $F_\rmS\colon \rmS(R)\to \Mod_{\rmS(R)}A$ be the free module functors. Then the actions
\begin{align*}
&\ast_R\colon\rmD(R)\times \rmS(R) \to \rmS(R),\\
&\ast_M\colon\Mod_{\rmD(R)}A \times \Mod_{\rmS(R)}A \to \Mod_{\rmS(R)}A,\\
&\ast_A\colon\rmD(A)\times \rmS(A) \to \rmS(A)
\end{align*}
are compatible with the equivalences $E_D$ and $E_S$: Let $X\in \rmD(R)$ and $Y\in \rmS(R)$. Then $E_S^{-1}(F_DX \ast_M F_SY)=E_S^{-1}F_S(X\ast_R Y)=E_D^{-1}F_DX \ast_A E_S^{-1}F_S Y$.
\end{prop}

\begin{proof}
Since $A\in \rmD(R)$ is a commutative separable ring object, $\Mod_{\rmD(R)}A$ admits a tensor product which defines an action on $\Mod_{\rmS(R)}A$; see~\Cref{cons:action-of-module-cats}. Applying $E_S^{-1}$ to the equality $F_S(X\ast_R Y)=F_DX\ast_M F_SY$, we obtain the first equality claimed in the statement. Let $\wt{X}$ be a $\rmK$-flat resolution of $X$. Since $A$ is flat, it holds that $A\ot_R \wt{X}=\wt{A\ot_R X}$ is a $\rmK$-flat resolution of $A\ot_R X$. We have
\begin{align*}
\res E_S^{-1}F_S(X\ast_R Y)&=\res\lambda_A(A\ot_R \wt{X}\ot_R Y)\\
&=A\ot_R \wt{X}\ot_R Y\\
&=\wt{X}\ot_R A\ot_R Y\\
&=\wt{X}\ot_R \res\lambda_A(A\ot_R Y)\\
&=\res((A\ot_R \wt{X}) \ot_A \lambda_A(A\ot_R Y))\\
&=\res((\wt{A\ot_R X})\ot_A\lambda_A(A\ot_R Y))\\
&=\res (E_D^{-1}F_D X \ast_A E_S^{-1}F_S Y).
\end{align*}
Since $\res$ is conservative, we have $E_S^{-1}F_S(X\ast_R Y)=E_D^{-1}F_DX\ast_A E_S^{-1}F_SY$.
\end{proof}

\section{\'Etale morphisms}\label{sec:finite-etale-morphisms}

In this section, we prove a singularity category analogue of~\cite[Theorem 3.5]{Balmer16} for finite \'etale morphisms of noetherian schemes and we also treat the case of open immersions.

\begin{thm}
\label[thm]{thm:singularity-etale-schemes}
Let $f\colon U\to X$ be a finite \'etale morphism of noetherian schemes. There is an equivalence $E\colon \rmS(U)\xr{\simeq} \Mod_{\rmS(X)}\rmR f_\ast \mcO_U$ of triangulated categories between the singularity category of $U$ and the category of modules in $\rmS(X)$ over the monad $\rmR f_\ast \mcO_U \ot_X -\colon \rmS(X)\to \rmS(X)$.
\end{thm}

\begin{proof}
Since $f$ is a flat morphism, the functor $f^\ast\colon \QCoh X\to \QCoh U$ is exact and its right adjoint $f_\ast\colon \QCoh U\to \QCoh X$ preserves injective quasi-coherent sheaves and is also exact because $f$ is finite. It holds that $\rmR f_\ast=f_\ast$ and in particular, $\rmR f_\ast \mcO_U=f_\ast \mcO_U$. Consider the derived adjoint triple
\[
\begin{tikzcd}[column sep=3em]
\rmD(X) \rar[shift left=3.5ex,"f^\ast"] \rar[shift left=2ex,phantom,"\perp"] \rar[shift right=2ex,phantom,"\perp"] \rar[shift right=3.5ex,"f^!"']& \lar["f_\ast"description] \rmD(U).
\end{tikzcd}
\]
We know by~\cite[Theorem 3.5]{Balmer16} that the adjunction $f^\ast \dashv f_\ast$ is monadic, so the functor $f_\ast\colon \rmD(U)\to \rmD(X)$ is identified with the forgetful functor $\Mod_{\rmD(X)}f_\ast f^\ast\to \rmD(X)$, therefore $f_\ast$ preserves compact objects. Combining~\cite[Proposition 3.2]{BalmerDellAmbrogioSanders16} with~\cite[Theorem 4.8]{Sanders22}, we deduce that $f^!=f^\ast$. Hence, $f^\ast \colon \QCoh X\to \QCoh U$ preserves injective quasi-coherent sheaves. Thus, we have an adjunction
\begin{equation}
\label{eq:etale-sing-adj}
\begin{tikzcd}
\rmS(X) \rar[shift left=5pt,"f^\ast"] \rar[phantom,"\perp"] & \lar[shift left=5pt,"f_\ast"] \rmS(U).
\end{tikzcd}
\end{equation}
Recall from~\Cref{ex:ring-objects} that $f_\ast \mcO_U\in \rmD(X)$ is a commutative separable ring object. Since $f_\ast$ is exact, $f_\ast \mcO_U\in \QCoh X$ is a flat quasi-coherent sheaf. Therefore, the action of $f_\ast \mcO_U$ on $\rmS(X)$ yields the separable monad $f_\ast \mcO_U \ot_X - \colon \rmS(X)\to \rmS(X)$. Since the derived adjunction
\begin{equation}
\label{eq:derived-adj}
\begin{tikzcd}
\rmD(X) \rar[shift left=5pt,"f^\ast"] \rar[phantom,"\perp"] & \lar[shift left=5pt,"f_\ast"] \rmD(U)
\end{tikzcd}
\end{equation}
satisfies the projection formula
\[
f_\ast(f^\ast E\ot_U^\rmL D)\cong E\ot^\rmL_X f_\ast D,\, \forall E\in \rmD(X),\, \forall D\in \rmD(U)
\]
and $f^\ast$ and $f_\ast$ are exact functors, the projection formula holds at the level of quasi-coherent sheaves. As a result, the monad $f_\ast \mcO_U\ot_X -\colon \rmS(X)\to \rmS(X)$ is realized by the adjunction~\eqref{eq:etale-sing-adj}. Since the counit of the derived adjunction~\eqref{eq:derived-adj} has a section, the counit of adjunction on quasi-coherent sheaves has a section, due to exactness of $f^\ast$ and $f_\ast$. It follows that the functor $f_\ast\colon \rmS(U)\to \rmS(X)$ is faithful. We conclude that the comparison functor $E\colon \rmS(U)\to \Mod_{\rmS(X)}f_\ast \mcO_U$ is an equivalence of triangulated categories; see~\Cref{summ:summary}.
\end{proof}

\begin{cor}
\label[cor]{cor:small-singularity-etale-schemes}
Let $f\colon U\to X$ be a finite \'etale morphism of noetherian schemes. There is an equivalence $E_{\mathrm{sg}}\colon \rmD_{\mathrm{sg}}(U)^\natural\xr{\simeq} \Mod_{\rmD_{\mathrm{sg}}(X)^\natural}\rmR f_\ast \mcO_U$ of triangulated categories.
\end{cor}

\begin{proof}
Again, note that $\rmR f_\ast \mcO_U=f_\ast \mcO_U$ due to exactness of $f_\ast$. The functors $f^\ast\colon \rmD^\rmb(\Coh X)\to \rmD^\rmb(\Coh U)$ and $f_\ast\colon \rmD^\rmb(\Coh U) \to \rmD^\rmb(\Coh X)$ preserve perfect complexes. Hence, we have an induced adjunction
\[
\begin{tikzcd}
\rmD_{\mathrm{sg}}(X)^\natural \rar[shift left=5pt,"f^\ast"] \rar[phantom,"\perp"] & \lar[shift left=5pt,"f_\ast"] \rmD_{\mathrm{sg}}(U)^\natural.
\end{tikzcd}
\]
We know by~\cite[Theorem 4.2]{Balmer16}, recalled in~\Cref{thm:modules-cg}, that $(\Mod_{\rmS(X)}f_\ast \mcO_U)^\rmc=\Mod_{\rmS(X)^\rmc}f_\ast \mcO_U$. Now restrict the equivalence $E\colon \rmS(U)\to \Mod_{\rmS(X)}f_\ast \mcO_U$, established in~\Cref{thm:singularity-etale-schemes}, on the subcategories of compact objects and note that, by~\cite[Lemma 5.4]{Stevenson14}, the functors $f_\ast$ and $f^\ast$ restricted on bounded complexes commute with the stabilization functors; see also the discussion preceding~\cite[Proposition 8.7]{Stevenson14}. We obtain the following commutative diagram providing the claimed equivalence:
\[
\begin{tikzcd}[/tikz/baseline=(\tikzcdmatrixname-\the\pgfmatrixcurrentrow-1.base)]
\rmS(U) \rar["\simeq"] & \Mod_{\rmS(X)}f_\ast \mcO_U
\\
\rmS(U)^\rmc \uar[hook] \rar["\simeq"] & \Mod_{\rmS(X)^\rmc}f_\ast \mcO_U \uar[hook]
\\
\rmD_{\mathrm{sg}}(U)^\natural \uar["\simeq"] \rar["E_{\mathrm{sg}}"',"\simeq"] & \Mod_{\rmD_{\mathrm{sg}}(X)^\natural}f_\ast \mcO_U. \uar["\simeq"'] 
\end{tikzcd}
\qedhere
\]
\end{proof}

\begin{rem}
The equivalence $E_{\mathrm{sg}}\colon \rmD_{\mathrm{sg}}(U)^\natural\xr{\simeq} \Mod_{\rmD_{\mathrm{sg}}(X)^\natural} \rmR f_\ast \mcO_U$, of~\Cref{cor:small-singularity-etale-schemes}, was independently obtained in recent work of~\cite{KarakikesKostas26} for proper schemes over a commutative noetherian ring. Their method utilizes the theory of intrinsic subcategories developed in~\cite{KostasPsaroudakisVitoria25}, via which they also obtain similar results for proper dg algebras and connective dg algebras in an equivariant context.
\end{rem}

\begin{prop}
\label[prop]{prop:open-immersion}
Let $f\colon U\to X$ be an open immersion of noetherian schemes. There is an equivalence $E\colon \rmS(U)\xr{\simeq} \Mod_{\rmS(X)}\rmR f_\ast \mcO_U$ of triangulated categories.
\end{prop}

\begin{proof}
Let $Z=X\setminus U$ and $\rmS_Z(X)=\{C\in \rmS(X)\mid \Supp(C)\subseteq Z\}$. The functor $f^\ast\colon \QCoh X\to \QCoh U$ preserves injectives and we obtain the adjunction
\[
\begin{tikzcd}
\rmS(X) \rar[shift left=5pt,"f^\ast"] \rar[phantom,"\perp"] & \lar[shift left=5pt,"f_\ast"] \rmS(U),
\end{tikzcd}
\]
which identifies $\rmS(U)$ with the localization $\rmS(X)/\rmS_Z(X)$ and $f^\ast$ with the quotient functor. All of this is explained in~\cite[Section 6]{Krause05}. By~\cite[Lemma 7.3]{Stevenson14}, $f^\ast$ commutes with the $\rmD(X)$-action. Further, $\rmD_Z(X)=\{C\in \rmD(X)\mid \Supp(C)\subseteq Z\}$ is a smashing tensor-ideal of $\rmD(X)$ and $\rmS_Z(X)=\rmD_Z(X)\ast \rmS(X)$ by~\cite[Corollary 4.11]{Stevenson13}. It follows that the associated localization functor is $\wt{\rmR f_\ast \mcO_U} \ot_X -\colon\rmS(X)\to \rmS(X)$, since it is given by acting with the right idempotent of $\rmD_Z(X)$, and we conclude that $f_\ast f^\ast=\wt{\rmR f_\ast \mcO_U}\ot_X - \colon \rmS(X)\to \rmS(X)$. It follows that the comparison functor $E\colon \rmS(U)\to \Mod_{\rmS(X)}\rmR f_\ast \mcO_U$ is an equivalence of triangulated categories.
\end{proof}

\begin{rem}
The analogous statement, as in~\Cref{prop:compatibility-of-actions}, also holds for finite \'etale morphisms and open immersions of noetherian schemes. We leave the details to the reader.
\end{rem}

\end{document}